\documentclass[11pt]{amsart}

\usepackage{latexsym}
\usepackage{amssymb}
\usepackage{amsmath}
\usepackage{color}

\newtheorem{theorem}{Theorem}[section]
\newtheorem{lemma}[theorem]{Lemma}
\newtheorem{proposition}[theorem]{Proposition}
\newtheorem{corollary}[theorem]{Corollary}

\newtheorem{remark}[theorem]{Remark}

\newcommand\supp{\mathop{\rm supp}}

\newcommand\ran{\mathop{\rm ran}}
\newcommand\id{\mathop{\rm id}}

\newcommand\nph{\varphi}
\newcommand\cb{\mathop{\rm cb}}

\newcommand\ccm{\mathop{\rm CC}}
\newcommand\cbm{\mathop{\rm CB}}

\newcommand{\cl}[1]{\mathcal{#1}}
\newcommand{\bb}[1]{\mathbb{#1}}

\begin{document}

\title[Completely compact Herz-Schur multipliers]{Completely compact Herz-Schur multipliers of dynamical systems}

\author[W. He]{Weijiao He}
\email{whe02@qub.ac.uk}
\address{Department of Mathematics, Taiyuan Normal University, Shanxi, China.}

\author[I. G. Todorov]{Ivan G. Todorov}
\email{todorov@udel.edu}
\address{Department of Mathematical Sciences, University of Delaware, 501 Ewing Hall, Newark, DE19716, USA}

\author[L. Turowska]{Lyudmila Turowska}
\email{turowska@chalmers.se}
\address{Department of Mathematical Sciences, Chalmers University
of Technology and the University of Gothenburg, Gothenburg SE-412 96, Sweden}


\date{}
\maketitle

\begin{abstract}
We prove that if $G$ is a discrete group and $(A,G,\alpha)$ is a C*-dynamical system such that
the reduced crossed product $A\rtimes_{r,\alpha} G$ possesses property (SOAP) then 
every completely compact Herz-Schur $(A,G,\alpha)$-multiplier can be approximated in the completely bounded norm by 
Herz-Schur $(A,G,\alpha)$-multipliers of finite rank. As a consequence, if $G$ has the approximation property (AP) then 
the completely compact Herz-Schur multipliers of $A(G)$ coincide with the closure of $A(G)$ in the completely bounded 
multiplier norm. We study the class of invariant completely compact Herz-Schur multipliers of $A\rtimes_{r,\alpha} G$
and provide a description of this class in the case of the irrational rotation algebra. 
\end{abstract}

\section{Introduction}\label{s_intro}

Fourier multipliers are transformations on function spaces associated with abelian topological groups
that act term-wise on the Fourier series of their elements. 
In the case of the group $\bb{Z}$, such mappings have a 
natural significance in classical Analysis, where
the approximation of a given function by trigonometric polynomials has been of paramount importance. 
When such truncations are realised as Fourier multipliers, these mappings have finite rank and are hence compact. 
In subsequent developments, 
compactness of Fourier multipliers has been studied from other perspectives as well, for example, 
in relation with the compactness of pseudo-differential operators (see e.g. \cite{cordes}).

For non-commutative locally compact groups $G$, where Fourier transform is not readily available, 
the natural setting for the study of Fourier multipliers is provided by the 
Fourier algebra $A(G)$ of the group $G$ -- a commutative Banach algebra consisting of all 
coefficients of the left regular representation of $G$ whose Gelfand spectrum can be canonically identified with $G$. 
This study was initiated in \cite{ch}, with illustrious subsequent history 
and some far-reaching applications to approximation techniques in operator algebra theory, 
where finite rank multipliers have played a cornerstone role (see \cite{brown-ozawa}). 
In \cite{lau}, Lau showed that the Fourier algebra $A(G)$ has a non-zero weakly
compact left multiplier if and only if $G$ is discrete and that, for 
discrete amenable groups, $A(G)$ coincides with the algebra of its weakly compact multipliers.
We refer the reader to \cite{Laug}, 
\cite{ESMAILVANDI} and \cite{Noncompactness} for further related results.

In the case of non-abelian groups, the new property of 
complete boundedness -- brought about by non-commutativity -- becomes a natural requirement, 
and the associated multipliers of $A(G)$ are known as Herz-Schur multipliers \cite{ch}.
Herz-Schur multipliers are related to Schur multipliers --
transformations on the algebra of all bounded operators on $L^2(G)$ 
that extend point-wise multiplication of integral kernels by a given fixed function -- 
via operator transference, pioneered in the area by Bo\.{z}ejko and Fendler \cite{bf}.
We refer the reader to \cite{t_survey}, \cite{tt_survey} for a survey of these results and ideas. 

A natural quantised version of the notion of compactness in the non-commutative setting -- complete compactness -- 
was  first introduced by Saar in his Diplomarbeit \cite{saar} under G. Wittstock's supervision 
and further developed  in unpublished paper \cite{webster} by Webster. More recently it was also studied in the context of operator multipliers in \cite{jltt} and completely almost periodic functionals on completely contractive Banach algebras in \cite{runde}. The notion has been important for the study of various operator space analogues of the Grothendieck approximation property and, in particular,  for questions revolving around finite rank approximation. 
The main aim of this paper is to initiate the study of complete compactness for Herz-Schur multipliers. 
We work in the higher generality of dynamical systems. 
Herz-Schur multipliers of crossed products were introduced, and 
their relation with the surrounding class of operator-valued Schur multipliers investigated, in \cite{mtt}
(see also \cite{bedos-conti} for discrete dynamical systems).
The completely compact operator-valued Schur multipliers were characterised in \cite{he}. 
Here we provide a characterisation of completely compact Herz-Schur multipliers of 
the reduced crossed product of (unital)  C*-algebras $A$ by the action of a discrete group $G$, 
under a mild approximation property.  

The paper is organised as follows.
In Section \ref{s_hsh}, we recall some results about operator-valued Schur multipliers, 
Herz-Schur multipliers of C*-dynamical systems $(A,G,\alpha)$ and their interrelation \cite{mstt} that will be needed in the sequel. 
In Section \ref{s_ccc}, we associate with every completely bounded map $\Phi$ with domain $A\rtimes_{\alpha,r} G$
a Herz-Schur multiplier $F_{\Phi}$, and identify some properties of the map $\Phi \to F_{\Phi}$.
This is the main technical tool, used in several subsequent results. 
We focus on actions of discrete groups, since the dynamical systems of non-discrete groups have no non-trivial 
compact Herz-Schur multipliers (see Proposition \ref{p_ccong}).  
Theorem \ref{th_ccc} provides a characterisation of completely compact 
Herz-Schur $(A,G,\alpha)$-multipliers in the case where $A\rtimes_{\alpha,r}G$ 
possesses the strong  operator approximation property (SOAP) \cite{er2}.
As a consequence, one can choose the maps, approximating the identity, to be in this case
Herz-Schur multipliers of finite rank. 
As an immediate corollary, 
we obtain that for groups with the approximation property (AP) \cite{haagerup-kraus}, a Herz-Schur multiplier 
$u : G\to\mathbb C$ defines a completely compact map $S_u:C_r^*(G)\to C_r^*(G)$ 
if and only if $u$ belongs to the closure $A_{\rm cb}(G)$ of $A(G)$ in the space
$M^{\rm cb}A(G)$ of Herz-Schur multipliers. 
Using a result of Bo\.{z}ejko \cite{bozejko}, 
we prove that the classes of compact and completely compact Herz-Schur multipliers $u : G\to\mathbb C$ 
are in general different. 

The rest of the paper is devoted to special classes of completely compact 
$(A,G,\alpha)$-multipliers. Namely, in Section \ref{s_ex} 
we examine the subclass of invariant multipliers of a C*-dynamical system $(A,G,\alpha)$, 
obtained by lifting completely bounded maps on the C*-algebra $A$ that satisfy a natural covariance property,
and exhibit a canonical way of constructing completely compact multipliers. 
We provide a description of the latter class in the case of the irrational rotation algebra. 
Finally, in Section \ref{s_K}, we consider completely compact multipliers of the dynamical system 
$(c_0(G),G,\alpha)$, where $\alpha$ is induced by left translations. 
Such multipliers induce, via the Stone-von Neumann Theorem, mappings on the space $\cl K$ 
of all compact operators on $\ell^2(G)$ that are characterised in terms of the Haagerup tensor product 
$\cl K\otimes\cl K$. In the opposite direction, we show that any compact Schur multiplier on $\cl K$ gives rise to 
a natural completely compact $(c_0(G),G,\alpha)$-multiplier.

 We finish this introduction with a general comment about 
a separability assumption. Many of our results rely on the development 
of the theory of Herz-Schur multipliers of C*-dynamical systems $(A,G,\alpha)$, undertaken in \cite{mtt}. 
In \cite{mtt}, the C*-algebra $A$ of the dynamical system is assumed to be separable. 
However, if the group $G$ of the dynamical system is discrete, as pointed out before \cite[Theorem 2.1]{mstt}, 
an inspection of the proofs from \cite{mtt} reveals that the separability assumption on $A$ can be lifted. 
In the present paper, all C*-dynamical systems $(A,G,\alpha)$ will be assumed to be over discrete groups and arbitrary C*-algebras $A$.


\section{Schur and Herz-Schur multipliers}\label{s_hsh}

We denote by $\cl B(H)$ the algebra of all bounded linear operators acing on a Hilbert space $H$, and 
by $I_H$ (or $I$ when $H$ is clear from the context) the identity operator on $H$.
For an operator space $\cl X\subseteq \cl B(H)$, we let $M_n(\cl X)$ be the space of all 
$n$ by $n$ matrices with entries in $\cl X$, and identify it with a subspace of $\cl B(H^n)$
(where $H^n$ is the direct sum of $n$ copies of $H$). 
If $\cl X$ and $\cl Y$ are operator spaces, acting on Hilbert spaces $H$ and $K$, respectively,
and $\nph : \cl X\to \cl Y$ is a linear map, we let as usual 
$\nph^{(n)} : M_n(\cl X)\to M_n(\cl Y)$ be the map given by $\nph^{(n)}\left((x_{i,j})_{i,j}\right) = \left(\nph(x_{i,j})\right)_{i,j}$. 
The map $\nph$ is called \emph{completely bounded} if 
$\left\|\nph\right\|_{\rm cb} := \sup_{n\in \bb{N}}\left\|\nph^{(n)}\right\| < \infty$.
We write ${\rm CB}(\cl X,\cl Y)$ for the space of all completely bounded maps from $\cl X$ into $\cl Y$. 
We let $\cl X\otimes_{\min}\cl Y$ be the \emph{minimal tensor product} of $\cl X$ and $\cl Y$, that is, the 
closure of the algebraic tensor product $\cl X\otimes\cl Y$ when considered as a subspace of $\cl B(H\otimes K)$. 
We will use throughout the paper basic results from operator space theory, and we refer the reader to the monographs
\cite{er, pa} for the necessary background.

For a locally compact group $G$, we let $\lambda^0$ be its left regular representation on $L^2(G)$; thus,
$$\lambda^0_t g(s) = g(t^{-1}s), \ \ \ g\in L^2(G), s,t\in G.$$
We use the same symbol, $\lambda^0$, to denote the left regular 
representation of $L^1(G)$ on $L^2(G)$.
Let
$$C_r^*(G) = \overline{\{\lambda^0(f) : f\in L^1(G)\}} \subseteq \cl B(L^2(G))$$
be the \emph{reduced C*-algebra} of $G$, 
${\rm VN}(G) := \overline{C_r^*(G)}^{w^*}$ the von Neumann algebra of $G$
(here $w^*$ denotes the weak* topology of $\cl B(L^2(G))$), 
and $A(G)$ be the \emph{Fourier algebra} of $G$, that is, 
the collection of the functions on $G$ of the form $s\to (\lambda^0_s\xi,\eta)$, where $\xi,\eta\in L^2(G)$.
The algebra $A(G)$ will be equipped with the operator space structure arising from its identification with the predual of 
${\rm VN}(G)$;
its norm will be denoted by $\|\cdot\|_{A}$, and by $\|\cdot\|_{A(G)}$ in cases where we need to emphasise the group.
 We also write $B(G)$ (resp. $B_r(G)$) for the Fourier-Stieltjes 
(resp. the reduced Fourier-Stieltjes) algebra of 
$G$. The space $B(G)$ (resp. $B_r(G)$) is generated by continuous positive-definite functions on $G$
(resp. by positive-definite functions weakly associated to the left regular representation $\lambda^0$ of $G$), 
and one has the inclusions $A(G)\subseteq B_r(G)\subseteq B(G)$. 
By \cite[Proposition 2.1]{eymard}, 
$B(G)$ and $B_r(G)$ can be identified with the dual space of the full C*-algebra $C^*(G)$ and reduced C*-algebra $C_r^*(G)$ of $G$, respectively. Moreover, when $B(G)$ is equipped with the norm
arising from the identification $B(G) = C^*(G)^*$, 
it becomes a Banach algebra with respect to the pointwise multiplication, 
and $A(G)$ and $B_r(G)$ are closed ideals of $B(G)$. 
The norms on $A(G)$ and $B_r(G)$ inherited from $B(G)$ 
coincide with the norms arising 
from the identifications $A(G)^* = {\rm VN}(G)$ and $B_r(G)=C_r^*(G)^*$.
We refer the reader to the monograph \cite{kl} for necessary further background from Abstract Harmonic Analysis.

A function $u : G\to \bb{C}$ is called a \emph{multiplier} of $A(G)$ if 
$uv\in A(G)$ for every $v\in A(G)$. 
We denote by $MA(G)$ the algebra of all multipliers of $A(G)$. 
An element $u\in MA(G)$ is called a \emph{Herz-Schur multiplier} of $A(G)$ \cite{ch} if 
the map $v\to uv$ on $A(G)$ is completely bounded
 (here, and in the sequel, we equip 
$A(G)$ and $B(G)$ with the operator space structures, arising from the identifications $A(G)^* = {\rm VN}(G)$ and $B(G) = C^*(G)^*$). 
We let $M^{\cb}A(G)$ be the algebra of all Herz-Schur multipliers of $A(G)$. We note that $u\in M^{\cb}A(G)$ if and only if the map $S_u: C_r^*(G)\to C_r^*(G)$, $\lambda^0(f)\mapsto\lambda^0(uf)$, $f\in L^1(G)$, is completely bounded, which is proved using similar arguments to the ones in \cite[Proposition 1.2]{ch}.

We henceforth fix a Hilbert space $H$ and a non-degenerate C*-algebra $A\subseteq \cl B(H)$.
Let $G$ be a discrete group
and $\alpha : G\to {\rm Aut}(A)$ be a point-norm continuous homomorphism; 
thus, $(A,G,\alpha)$ is a C*-dynamical system.
We write $\{\delta_s : s\in G\}$ for the canonical orthonormal basis of $\ell^2(G)$.  
We let $\ell^1(G,A)$ be the convolution *-algebra of all summable functions $f : G\to A$,
set $\cl H := \ell^2(G)\otimes H$ and identify it with the Hilbert space $\ell^2(G,H)$
of all square summable $H$-valued functions on $G$. 
Let $\lambda : G\to \cl B(\cl H)$, $t\to \lambda_t$, be the unitary representation of $G$
given by 
$$\lambda_t \xi(s) = \xi(t^{-1}s), \ \ \ s,t\in G, \xi\in \cl H;$$
note that $\lambda_t = \lambda_t^0 \otimes I$.
Let $\pi : A\to \cl B(\cl H)$ be the *-representation given by
$$\pi(a)\xi(s) = \alpha_{s^{-1}}(a)(\xi(s)), \ \ \ a\in A, \xi\in \cl H, s\in G.$$
We note the covariance relation
\begin{equation}\label{eq_pil}
\pi(\alpha_t(a)) = \lambda_t \pi(a) \lambda_t^*, \ \ \ a\in A, t\in G.
\end{equation}
The pair $(\pi,\lambda)$ gives rise to a *-representation
$\tilde{\pi} : \ell^1(G,A) \to \cl B(\cl H)$, given by
\begin{equation}\label{eq_repel}
\tilde{\pi}(f) = \sum_{s\in G} \pi(f(s))\lambda_s, \ \ \ f\in \ell^1(G,A).
\end{equation}
(Note that the series on the right hand side of (\ref{eq_repel}) converges in norm for every $f\in \ell^1(G,A)$).
The reduced crossed product $A\rtimes_{\alpha,r} G$ is defined by letting
$$A\rtimes_{\alpha,r} G = \overline{\tilde{\pi}(\ell^1(G,A))},$$
where the closure is taken in the operator norm of $\cl B(\cl H)$.
Note that, after identifying $A$ with $\pi(A)$, we may consider $A$ as a C*-subalgebra of $A\rtimes_{\alpha,r} G$.
It is well-known that 
 if $\rho : A\to \cl B(K)$ is a faithful non-degenerate *-representation and $\alpha'$ is the canonical action of $G$ on $\rho(A)$, arising from $\alpha$, then
$A\rtimes_{\alpha,r} G\cong \rho(A)\rtimes_{\alpha',r} G$ canonically
(see e.g. \cite[Theorem 7.7.5]{ped}).

Identifying $\cl H$ with $\oplus_{s\in G} H$, we associate to every operator $x\in \cl B(\cl H)$
a matrix $(x_{p,q})_{p,q}$, where $x_{p,q}\in \cl B(H)$, $p,q\in G$; thus, 
$$\langle x_{p,q} \xi,\eta\rangle = \langle x (\delta_q \otimes \xi), \delta_p \otimes \eta\rangle, \ \ \ 
\xi,\eta\in H, p,q\in G.$$
In particular, if $a\in A$ and $t\in G$ then 
\begin{equation}\label{eq_piapq}
\left(\pi(a)\lambda_t\right)_{p,q} 
=
\begin{cases}
\alpha_{p^{-1}}(a) & \text{if } pq^{-1} = t\\
0 & \text{if } pq^{-1} \neq t.
\end{cases}
\end{equation}
Equation (\ref{eq_piapq}) implies that 
$$\left(\pi(a)\lambda_t\right)_{e,qp^{-1}} = \delta_{t, pq^{-1}} a$$
(here $\delta_{s,r}$ is the Kronecker symbol).
By linearity and continuity, 
\begin{equation}\label{translationformula} 
x_{p,q} = \alpha_{p^{-1}} (x_{e, qp^{-1}}), \ \ \ x\in A\rtimes_{\alpha,r} G.
\end{equation}

Let $\cl E : \cl B(\ell^2(G))\otimes_{\min} A \to A$ be the (unital completely positive) map given by 
$\cl E(x) = x_{e,e}$; note that $\cl E(\pi(a)) = a$, $a\in A$. 
Equation (\ref{eq_piapq}) implies that $x_{e,r} = \cl E(x\lambda_r)$, $x\in A\rtimes_{\alpha,r} G$, $r\in G$. 
Now (\ref{translationformula}) can be rewritten as
\begin{equation}\label{eq_T}
x_{p,q} = \alpha_{p^{-1}}(\cl E(x\lambda_{qp^{-1}})), \ \ \ x\in A\rtimes_{\alpha,r} G, \ p,q\in G.
\end{equation}

To every operator $x\in A\rtimes_{\alpha,r} G$, one can associate the family $(a_s)_{s\in G}\subseteq A$, where
$a_t = \cl E(x\lambda_{t^{-1}})$; we call $\sum_{t\in G} \pi(a_t) \lambda_t$ the \emph{Fourier series} of $x$ 
and write $x\sim \sum_{t\in G} \pi(a_t) \lambda_t$ (no convergence is assumed).
Equation (\ref{eq_T}) shows that 
\begin{equation}\label{eq_melc}
x\sim \sum_{t\in G}\pi(a_t)\lambda_t \ \Longrightarrow \ x_{p,q} = \alpha_{p^{-1}}(a_{pq^{-1}}).
\end{equation}
Thus, if $x\in A\rtimes_{\alpha,r} G$ then the diagonal of its matrix coincides with the family
$(\alpha_{r^{-1}}(\cl E(x)))_{r\in G}$.

We note that
\begin{equation}\label{eq_el}
\alpha_{t}(\cl E(x))=\cl E(\lambda_t x\lambda_t^*), \ \ \ t\in G, x\in A\rtimes_{\alpha,r} G.
\end{equation}
The latter equality follows from (\ref{eq_pil}) 
in the case where 
$x = \pi(a) \lambda_s$ and follows by linearity and continuity for a general $x\in A\rtimes_{\alpha,r} G$.

If $F : G\to {\rm CB}(A)$ is a bounded map and $f\in \ell^1(G,A)$, let $F\cdot f\in  \ell^1(G,A)$ be the function given by
$$(F\cdot f)(t) = F(t)(f(t)), \ \ \ t\in G.$$
Recall  \cite[Definition 3.1]{mtt} that $F$ is called a \emph{Herz-Schur $(A,G,\alpha)$-multiplier}
if the map $S_F$, given by
\begin{equation}\label{eq_dehs}
S_F(\tilde{\pi}(f)) = \tilde{\pi}(F\cdot f), \ \ \ f\in \ell^1(G,A),
\end{equation}
is completely bounded.
If $F$ is a Herz-Schur $(A,G,\alpha)$-multiplier, then $S_F$ has a (unique) extension to a
completely bounded map on $A\rtimes_{\alpha,r} G$, which will be denoted in the same way.
We write $\frak{S}(A,G,\alpha)$ for the space of all Herz-Schur $(A,G,\alpha)$-multipliers, 
and set $\|F\|_{\rm m} = \|S_F\|_{\rm cb}$, $F\in \frak{S}(A,G,\alpha)$. 
Note that, if $u\in M^{\cb}A(G)$ and $\tilde{u} : G\to {\rm CB}(A)$ is the function given by 
$\tilde{u}(t) = u(t)\id_A$, then $\tilde{u}\in \frak{S}(A,G,\alpha)$ and its norm in $\frak{S}(A,G,\alpha)$
coincides with its norm in $M^{\cb}A(G)$ \cite[Remark 3.2 (ii)]{mtt}. 
Without risk of confusion, for brevity we will denote by $S_u$ the map $S_{\tilde{u}}$.

Let $\nph : G\times G\to {\rm CB}(A,\cl B(H))$ be a bounded function and 
$S_{\nph}$ be the map from the space of all $G\times G$ $A$-valued matrices into the space
of all $G\times G$ $\cl B(H)$-valued matrices, given by 
$$S_{\nph}\left((x_{p,q})_{p,q}\right) = \left(\nph(q,p)(x_{p,q})\right)_{p,q}.$$
The map $\nph$ is called a \emph{Schur $A$-multiplier}  if 
$S_{\nph}$ is a completely bounded map from $\cl B(\ell^2(G)) \otimes_{\min} A$ into 
$\cl B(\cl H)$. 
 By the paragraph after \cite[Theorem 2.1]{mstt}, this 
definition is equivalent, for a discrete group $G$, to the definition 
of Schur $A$-multipliers in 
\cite[Definition 2.1]{mtt}. 
For a Schur $A$-multiplier $\nph$, we write $\|\nph\|_{\rm m} = \|S_{\nph}\|_{\rm cb}$.
Schur $\bb{C}$-multipliers are referred to as \emph{Schur multipliers}.
For a bounded function $F : G\to {\rm CB}(A)$, let
$\mathcal{N}(F): G \times G \to {\rm CB}(A)$ be the function given by
\[ \mathcal{N}(F)(s,t) (a) = \alpha_{t^{-1}}(F(ts^{-1})(\alpha_t(a))), \;\;\;\; s,t \in G, \ a \in A.\]
In the case where $A = \bb{C}$, we write $N(u)$ in the place of $\cl N(u)$. 
The following statement \cite[Theorem 3.18]{mtt} is a crossed product version of a classical 
result of Bo\.{z}ejko and Fendler \cite{bf} in the case $A = \bb{C}$.

\begin{theorem}\label{th_tra}
The map $\cl N$ is an isometric injection
from $\frak{S}(A,\alpha,G)$ into the algebra of Schur $A$-multipliers. 
\end{theorem}

 Note that the image $\cl N(\frak{S}(A,\alpha,G))$ coincides with the so-called invariant Schur-$A$-multipliers, denoted by $\frak S_{\rm inv}(G,G,A)$ in \cite[p. 413]{mtt}.

 We remark that in \cite{mtt} Herz-Schur $(A,G,\alpha)$-multipliers and Schur $A$-multipliers were defined for general locally compact group $G$ and separable C*-algebra $A$. But in the case of discrete $G$  the separability condition on $A$ can be removed without changing the statements from \cite{mtt} 
(see the comment before \cite[Theorem 2.1]{mstt}).

\smallskip

The following observation will be frequently used: 
 If $G$ is discrete then $S_{\cl N(F)}|_{A\rtimes_{\alpha,r}G}=S_F$ for any $F\in\frak{S}(A,\alpha,G)$. 
To see this,  it suffices to  show that 
$$S_{\cl N(F)}(\pi(a)\lambda_s)=\pi(F(s)(a))\lambda_s, \ \ \ a\in A, s\in G.$$
Write $\pi(a)\lambda_s=(x_{p,q})_{p,q}$ for $x_{p,q}=\delta_{s,pq^{-1}}\alpha_{p^{-1}}(a)$ (see (\ref{eq_piapq}))
and 
\begin{eqnarray*}
(S_{\cl N(F)}(\pi(a)\lambda_s)_{p,q}
& = &
\delta_{s,pq^{-1}}\cl N(F)(q,p)(\alpha_{p^{-1}}(a))\\
& =& 
\delta_{s,pq^{-1}}\alpha_{p^{-1}}(F(pq^{-1})(a))\\
& = &  
\delta_{s,pq^{-1}}\alpha_{p^{-1}}(F(s)(a))
= 
(\pi(F(s)(a))\lambda_s)_{p,q}.
\end{eqnarray*}


\section{Characterisation of complete compactness}\label{s_ccc}

Let $\cl X$ and $\cl Z$ be  operator spaces. A completely  bounded map $\Psi : \cl X \to \cl Z$ is called
\emph{completely compact} if for very $\epsilon > 0$ there exists a finite dimensional
subspace $\cl Y\subseteq \cl Z$ such that
\begin{equation}\label{eq_psicc}
{\rm dist}(\Psi^{(m)}(x), M_m(\cl Y)) < \epsilon, \mbox{ for all } x\in M_m(\cl X), \|x\| \leq 1, \mbox{ and all } m\in \bb{N}.
\end{equation}
We denote by $\ccm(\cl X, \cl Z)$ the space of all completely compact linear maps from $\cl X$ to $\cl Z$, and write $\ccm(\cl X)$ for $\ccm (\cl X,\cl X)$. 
By ${\rm F}(\cl X)$ we denote the subspace of all (linear) maps of finite rank on $\cl X$. 
Clearly, ${\rm F}(\cl X)\subseteq \ccm (\cl X)$ and any completely compact map  on $\cl X$ is compact.

A net $(\Phi_i)_{i\in \bb{I}} \subseteq \cbm(\cl X)$ is said to converge to 
$\Phi\in\cbm(\cl X)$  in the {\it strongly stable point norm topology} \cite[p. 197]{er2}
if
$$\left\|(\id\otimes \Phi_i)(x) - (\id\otimes\Phi)(x)\right\|\to_{i\in \bb{I}} 0, \ \ \ x\in \cl B(K)\otimes_{\min} \mathcal X,$$
for any Hilbert space $K$.
An operator space $\mathcal X$ is said to possess the {\it strong operator approximation property (SOAP)}
if the identity map on $\mathcal X$ can be approximated by finite rank maps in the strongly stable point norm topology.
 The notion was introduced by Effros and Ruan in \cite{er2} as an operator space analogue of the approximation property for Banach spaces. It admits a characterisation using complete compactness. Namely, $\cl X$ has the SOAP if for any operator space $\cl Z$ and any completely compact map $\Psi\in \ccm(\cl Z,\cl X)$, $\Psi$ can be approximated by finite rank maps completely uniformly, i.e. in  the $\|\cdot\|_{\cb}$-norm, see \cite[Theorem 5.9]{webster} or \cite[Proposition 4.4.6]{webster_dis}.
In particular, if $\mathcal X$ has the SOAP then  
$\overline{{\rm F}(\mathcal X)}^{\rm cb} = {\rm CC}(\mathcal X)$.  

\begin{lemma}\label{l_soat}
Let $\cl X$ be an operator space, 
$\Psi \in \ccm(\cl X)$, $\Phi\in \cbm(\cl X)$ and  $(\Phi_{i})_{i\in \bb{I}}\subseteq \cbm(\cl X)$ be a
net such that $\Phi_i\to_{i\in \bb{I}} \Phi$ in the strongly stable point-norm topology. 
Then $\|\Phi_i \circ \Psi - \Phi \circ \Psi\|_{\cb}\to_{i\in \bb{I}} 0$. 
\end{lemma}

\begin{proof}
By \cite[Proposition 5.6]{webster} $\Phi_i\to_{i\in \bb{I}} \Phi$ 
completely uniformly on completely compact sets of $\mathcal X$ 
(see \cite[Definitions 5.1 and 5.2]{webster} for the terminology).  If $\Psi$ is completely compact then 
$(\{\Psi^{(n)}(x) : x\in M_n(\cl X), \|x\|\leq 1\})_{n\in \bb{N}} $ is a subset of a completely compact set and hence 
$\|\Phi_i\circ\Psi-\Phi\circ\Psi\|_{\rm cb}\to_{i\in \bb{I}} 0$.
\end{proof}

The following observation shows that, when studying compact or completely compact
multipliers, the interest lies only in discrete groups.

\begin{proposition}\label{p_ccong}
Let $G$ be a non-discrete locally compact group and $u : G\to \bb{C}$ be a Herz-Schur multiplier.
If the map $S_u : C^*_r(G)\to C^*_r(G)$ is compact then $u = 0$.
\end{proposition}

\begin{proof}
Suppose that $S_u$ is compact.
Let $B_r(G)$ be the reduced Fourier-Stieltjes algebra of $G$.
Identifying the dual $C^*_r(G)^*$ of $C^*_r(G)$ with $B_r(G)$, we have that the adjoint 
map $S_u^* : B_r(G)\to B_r(G)$
is compact (see e.g. \cite[Theorem 1.4.4]{murphy}).  Note that $S_u^*(v) = uv$, $v\in B_r(G)$.
Since $A(G)\subseteq B_r(G)$,  $A(G)$ is the closed hull of its compactly supported elements and $A(G)\cap C_c(G)=B_r(G)\cap C_c(G)$ (see \cite[Proposition 2.3.3]{kl}), the map $S_u^*$ leaves $A(G)$ invariant. The restriction of $S_u^*$ to $A(G)$ is thus compact, 
and by \cite[Proposition 6.9]{lau}, $u = 0$.
\end{proof}

We henceforth assume that the group $G$ is discrete. Recall that $A\subseteq \cl B(H)$ is a 
fixed non-degenerate C*-algebra, equipped with an action $\alpha : G\to {\rm Aut}(A)$. 
It will be convenient to set 
$$\cl V_{A,G,\alpha} = {\rm span}\left\{\pi(a)\lambda_s : a\in A, s\in G\right\};$$
thus, $\cl V_{A,G,\alpha}$ is a (dense) subalgebra of $A\rtimes_{\alpha,r} G$. 
We will write 
$$\frak{S}_{\rm cc}(A,G,\alpha) = \left\{F\in \frak{S}(A,G,\alpha) : S_F \mbox{ is completely compact}\right\}.$$
If $\Phi : A\rtimes_{\alpha,r} G \to \cl B(\ell^2(G))\otimes_{\min} A$ is a bounded linear map, 
we let $F_\Phi : G\to \cl B(A)$ be the function given by
\begin{equation}\label{eq_hPhi}
F_\Phi(s)(a) = \cl E(\Phi(\pi(a)\lambda_s)\lambda_s^*).
\end{equation}

\begin{proposition}\label{p_FPhi}
Let $\Phi : A\rtimes_{\alpha,r} G \to \cl B(\ell^2(G))\otimes_{\min} A$ be a completely bounded map. 
\begin{itemize}
\item[(i)] The map $F_{\Phi}$ is a Herz-Schur $(A,G,\alpha)$-multiplier and $\|F_{\Phi}\|_{\rm m}\leq \|\Phi\|_{\rm cb}$;
\item[(ii)] If $F \in \frak{S}(A,G,\alpha)$ 
then $F_{S_F} =  F$; 
\item[(iii)] If $\Phi$ completely compact then $F_{\Phi}(s)$ is completely compact for every $s\in G$;
\item[(iv)] If $\Phi$ has finite rank and ${\rm ran}\Phi \subseteq \cl V_{A,G,\alpha}$ then 
$F_{\Phi}(s)$ has finite rank for every $s\in G$ and $S_{F_\Phi}\in {\rm F}(A\rtimes_{\alpha,r}G)$; 
\item[(v)] If $\nph$ is a Schur $A$-multiplier and $\Psi$ is the restriction of $S_{\nph}$ to $A\rtimes_{\alpha,r} G$
then $F_{\Psi}(s) = \nph(s^{-1},e)$, $s\in G$.  
\end{itemize}
\end{proposition}

\begin{proof}
(i) By the Haagerup-Paulsen-Wittstock Theorem (see e.g. \cite[Theorem 8.4]{pa}),
there exist a Hilbert space $K$, operators $V,W :  H \to K$ and a *-representation
$\rho : A\rtimes_{\alpha,r} G \to \cl B(K)$ such that
\begin{equation}\label{eq_hpw}
\Phi(x) = W^*\rho(x)V, \ \ \ x\in A\rtimes_{\alpha,r} G.
\end{equation}
Using (\ref{eq_el}), we have
\begin{eqnarray}\label{eq_5}
\cl N\left(F_{\Phi}\right)(s,t)(a)
& = &
\alpha_{t^{-1}}\left(F_{\Phi}(ts^{-1})(\alpha_t(a))\right)\\
& = &
\alpha_{t^{-1}}\left(\cl E(\Phi(\pi(\alpha_t(a))\lambda_{ts^{-1}})\lambda_{ts^{-1}}^*)\right)\nonumber\\
& = &
\alpha_{t^{-1}}\left(\cl E(\Phi(\lambda_t\pi(a)\lambda_{t^{-1}}\lambda_{ts^{-1}})\lambda_{ts^{-1}}^*)\right)\nonumber\\
& = &
\cl E\left(\lambda_{t^{-1}} \Phi((\lambda_t\pi(a)\lambda_{s^{-1}})\lambda_s\lambda_{t^{-1}})\lambda_t\right)\nonumber\\
& = &
\cl E\left(\lambda_{t^{-1}} \Phi(\lambda_t\pi(a)\lambda_{s^{-1}})\lambda_s\right)\nonumber\\
& = &
\cl E\left(\lambda_t^* W^* \rho(\lambda_t)\rho(\pi(a))\rho(\lambda_s)^*V\lambda_s\right).\nonumber
\end{eqnarray}
Note that, if $E_{p,q}$ denotes the matrix unit in $\cl B(\ell^2(G,H))$ 
with $I$ at the $(p,q)$-entry, $p,q\in G$, and
zero elsewhere, then $\cl E(x) = E_{e,e} x E_{e,e}$, $x\in A\rtimes_{r,\alpha} G$.
Let 
$\tilde{V}(s) = \rho(\lambda_s)^*V\lambda_s E_{e,e}$ and $\tilde{W}(t) = \rho(\lambda_t)^*W\lambda_t E_{e,e}$,
$s,t\in G$; thus, $\tilde{V}, \tilde{W} : G\to \cl B(H, K)$, and 
\begin{equation}\label{eq_VW}
\sup_{s\in G}\|\tilde{V}(s)\| \sup_{t\in G}\|\tilde{W}(t)\| \leq \|V\|\|W\|.
\end{equation}
By (\ref{eq_5}), 
\begin{equation}\label{eq_Nst}
\cl N(F_{\Phi})(s,t)(a) = \tilde{W}(t)^*(\rho\circ\pi)(a) \tilde{V}(s), \ \ \ a\in A, s,t\in G.
\end{equation}
By \cite[Theorems 2.6 and 3.8]{mtt}, $F_{\Phi}$ is a Herz-Schur $(A,G,\alpha)$-multiplier.
Equations (\ref{eq_VW}),  (\ref{eq_Nst}) and Theorem \ref{th_tra} show that 
$$\|F_{\Phi}\|_{\rm m} = \|S_{\cl N(F_{\Phi})}\|_{\rm cb} \leq \|V\|\|W\|.$$
Taking the infimum over all possible choices of $V$ and $W$ in the representation (\ref{eq_hpw}) of $\Phi$, 
we obtain that $\|F_{\Phi}\|_{\rm m} \leq \|\Phi\|_{\rm cb}$.

(ii) Let $F \in \frak{S}(A,G,\alpha)$. 
The fact that $F_{S_F} = F$ follows from the definition (\ref{eq_dehs}) of $S_F$ and the definition (\ref{eq_hPhi}) of $F_{\Phi}$.

(iii) follows from the definition (\ref{eq_hPhi}) of $F_\Phi$ and the fact that $\cl E$ is completely bounded.

(iv) Since $\Phi$ has finite rank, there exists a finite subset $E\subseteq G$ and 
a finite dimensional subspace $\cl U\subseteq A$, $s\in E$, such that 
$${\rm ran}\Phi \subseteq {\rm span}\{\pi(a)\lambda_s : s\in E, a\in \cl U\}.$$
A direct verification shows that 
\begin{equation*}
{\rm ran}F_{\Phi}(s) \subseteq \left\{ 
\begin{array}{ll}
\cl U & \text{if } s\in E\text{,} \\ 
\{0\} & \text{if } s\not\in E.
\end{array}
\right.
\end{equation*}
Clearly, ${\rm ran}S_{F_\Phi} \subseteq {\rm span}\{\pi(a)\lambda_s : s\in E, a\in \cl U\}.$ 

(v) 
According to (\ref{eq_piapq}),  
$$\left(\Psi(\pi(a)\lambda_s)\lambda_{s^{-1}}\right)_{p,q} 
= 
(S_{\nph}(\pi(a)\lambda_s)_{p,s^{-1}q} 
= 
\begin{cases}
\nph(s^{-1}q,p)(\alpha_{p^{-1}}(a)) & \text{if } p = q\\
0 & \text{if } p \neq q.
\end{cases}
$$
Thus, 
$$F_{\Psi}(s)(a) = \left(\Psi(\pi(a)\lambda_s)\lambda_{s^{-1}}\right)_{e,e} = \nph(s^{-1},e)(a), \ \ \ a\in A.$$
\end{proof}

Proposition \ref{p_FPhi} (ii) and (iii) give the following immediate corollary.

\begin{corollary}\label{c_Fscc}
If $F\in \frak{S}_{\rm cc}(A,G,\alpha)$ then $F(s)$ is completely compact for every $s\in G$;
\end{corollary}

We will call a (possibly vector-valued) function $\nph$ defined on $G\times G$
\emph{band finite} if there exists a finite set $E\subseteq G$ such that
$\nph(s,t) = 0$ if $ts^{-1}\not\in E$.
Let
$V : \ell^2(G)\otimes H\to\ell^2(G)\otimes\ell^2(G)\otimes H$
be the isometry given by 
$$V(\delta_s\otimes \xi) = \delta_s\otimes\delta_s\otimes \xi, \ \ \ s\in G, \xi\in H,$$
and
$\tau : A\rtimes_{r,\alpha}G \to C_r^*(G)\otimes_{\rm min} (A\rtimes_{r,\alpha}G)$ be the  dual co-action   to $\alpha$, 
that is, the *-homomorphism, given by
$$\tau\left(\pi(a)\lambda_t\right) = \lambda_t^0\otimes\pi(a)\lambda_t, \ \ \ a\in A, t\in G,$$
(see e.g. \cite{imai_takai}). 

\begin{lemma}\label{l_V}
The following hold:
\begin{itemize}
\item[(i)] $V^*\tau(x)V = x$, $x\in A\rtimes_{r,\alpha}G$;
\item[(ii)] If $\Phi\in {\rm CB}(A\rtimes_{r,\alpha}G)$ then
\begin{equation}\label{eq_Snph}
S_{F_{\Phi}}(x) = V^*(\id\otimes\Phi)(\tau(x))V, \ \ \  x\in A\rtimes_{r,\alpha}G.
\end{equation}
\end{itemize}
\end{lemma}

\begin{proof}
(i) 
If $a\in A$, $s,r,t\in G$ and $\xi,\eta\in H$ then 
\begin{eqnarray*}
& & 
\left\langle V^*\tau(\pi(a)\lambda_t)V(\delta_s\otimes\xi),\delta_r\otimes\eta\right\rangle\\
& = & 
\left\langle (\lambda_t\otimes \pi(a)\lambda_t)(\delta_s\otimes\delta_s\otimes\xi),\delta_r\otimes\delta_r\otimes\eta\right\rangle\\
& = & 
\left\langle \delta_{ts}\otimes \pi(a)\lambda_t(\delta_s\otimes\xi),\delta_r\otimes\delta_r\otimes\eta\right\rangle
= 
\left\langle \pi(a)\lambda_t(\delta_s\otimes\xi),\delta_r\otimes\eta\right\rangle,
\end{eqnarray*}
by (\ref{eq_piapq}).

(ii)
Let $a\in A$, $s,r,t\in G$ and $\xi,\eta\in H$. Using (\ref{eq_piapq}) and (\ref{eq_T}), we have
\begin{eqnarray*}
& & 
\langle V^*(\id\otimes\Phi)(\tau(\pi(a)\lambda_t))V(\delta_s\otimes\xi),\delta_r\otimes\eta\rangle\\
& = & 
\langle (\id\otimes\Phi)(\lambda_t^0 \otimes \pi(a)\lambda_t)(\delta_s\otimes \delta_s\otimes\xi),\delta_r\otimes \delta_r\otimes\eta\rangle\\
& = & 
\langle \left(\lambda_t^0\otimes \Phi(\pi(a)\lambda_t)\right)(\delta_s\otimes \delta_s\otimes\xi),\delta_r\otimes \delta_r\otimes\eta\rangle\\
& = & 
\langle \delta_{ts}\otimes \Phi(\pi(a)\lambda_t)(\delta_s\otimes\xi),\delta_r\otimes \delta_r\otimes\eta\rangle\\
& = & 
\left\{ 
\begin{array}{ll}
0 & \text{if } ts\neq r\text{,} \\ 
\langle \Phi(\pi(a)\lambda_t)(\delta_s\otimes\xi),\delta_r\otimes\eta\rangle & \text{if } ts = r.
\end{array}
\right.\\
& = & 
\left\{ 
\begin{array}{ll}
0 & \text{if } ts\neq r\text{,} \\ 
\left\langle \alpha_{(ts)^{-1}}\left(\cl E(\Phi(\pi(a)\lambda_t)\lambda_t^*)\right)\xi,\eta\right\rangle   & \text{if } ts = r.
\end{array}
\right.\\
& = & 
\left\langle \pi\left(\cl E(\Phi(\pi(a)\lambda_t)\lambda_t^*)\right)(\delta_{ts}\otimes\xi),\delta_r\otimes\eta\right\rangle\\
& = & 
\left\langle S_{F_{\Phi}}\left(\pi(a)\lambda_t\right)(\delta_s\otimes\xi),\delta_r\otimes\eta\right\rangle.
\end{eqnarray*}
\end{proof}

\begin{lemma}\label{soat}
Let $G$ be a discrete group and 
$(A,G,\alpha)$ be a C*-dynamical system such that $A\rtimes_{r,\alpha}G$ has the SOAP. 
Then there exists a net $(F_i)_{i\in \bb{I}}$ of finitely supported Herz-Schur $(A,G,\alpha)$-multipliers 
with $F_i(s)\in {\rm F}(A)$, $i\in \bb{I}$, $s\in G$, such that 
$(S_{F_i})_{i\in \bb{I}}$ converges to the identity map on 
$A\rtimes_{r,\alpha}G$ in the strongly stable point norm topology.
\end{lemma}

\begin{proof}
Write $\mathfrak A = A\rtimes_{r,\alpha}G$ and let
$(\Phi_i)_{i\in \bb{I}} \subseteq {\rm F}(\mathfrak A)$ be a net such that 
$$\|(\id\otimes \Phi_i)(x) - x\|\to_{i\in \bb{I}} 0, \ \ \  x\in \cl B(\ell^2)\otimes_{\rm min}\mathfrak A.$$
Following the proof of \cite[Theorem 4.3]{mckee}, 
given $\varepsilon > 0$, for each $i\in \bb{I}$ there exists 
a finite rank map $\Phi_{i,\varepsilon}\in {\rm CB}(\mathfrak A)$ whose range lies in $\cl V_{A,G,\alpha}$, such that
$$\|\Phi_i-\Phi_{i,\varepsilon}\|_{\rm cb} < \varepsilon, \ \  \ i\in \bb{I}.$$
Hence
$$\|(\id\otimes \Phi_{i,\varepsilon})(x) - x\| \to_{(i,\epsilon)} 0, \ \ \  x\in\cl B(H)\otimes_{\rm min}\mathfrak A,$$
along the product directed set.

The function $F_{i,\varepsilon} := F_{\Phi_{i,\varepsilon}}$ is finitely supported. By Proposition \ref{p_FPhi} (i) and (iv), 
$F_{i,\varepsilon}$ is a Herz-Schur multiplier and  
$F_{i,\varepsilon}(s)\in {\rm F}(A)$, $s\in G$, $i\in \bb{I}$. 
If $y\in \cl B(H)\otimes_{\rm min} \mathfrak A$, 
using Lemma \ref{l_V} we have 
\begin{eqnarray*}
\|\id\otimes S_{F_{i,\varepsilon}}(y) - y\|
& = &
\|(\id\otimes V^*)(\id\otimes\id\otimes\Phi_{i,\varepsilon}(\id\otimes\tau(y))(\id\otimes V)\\&&- (\id\otimes V^*)(\id\otimes\tau(y)(\id\otimes V)\|\\
&\leq &\|(\id\otimes \id\otimes\Phi_{i,\varepsilon})((\id\otimes\tau)(y)) - (\id\otimes\tau)(y)\|
\to_{(i,\epsilon)} 0,
\end{eqnarray*}
that is, $S_{F_{i,\varepsilon}}$ converges to the identity map in the strongly stable point norm topology.
\end{proof}


\begin{theorem}\label{th_ccc}
Let $F$ be a Herz-Schur multiplier of $(A,G,\alpha)$. 
Assume that $A\rtimes_{r,\alpha}G$ possesses the SOAP. 
The following are equivalent:
\begin{itemize}
\item[(i)] the map $S_F$ is completely compact;

\item[(ii)] 
there exist 
a net $(F_i)_{i\in \bb{I}}\subseteq \mathfrak S(A,G,\alpha)$ and a finite dimensional subspace $\cl U_{i}\subseteq A$, 
such that $F_i$ is finitely supported, $\ran F_i(s)\subseteq \cl U_i$
for each $i\in \bb{I}$ and each $s\in G$, 
and $\|F_i - F\|_{\rm m}\to_{i\in \bb{I}} 0$;

\item[(iii)] there exists a net $(\nph_i)_{i\in \bb{I}}$ of band finite Schur $A$-multipliers
and a finite dimensional subspace $\cl U_{i}\subseteq A$, $i\in \bb{I}$,
such that
$\ran \alpha_s \circ \nph_i(s,e) \subseteq \cl U_{i}$ for all $i\in \bb{I}$ and all $s\in G$,
and
$\|S_{\cl N(F)} - S_{\nph_i}\|_{\cb} \to_{i\in \bb{I}} 0$.
\end{itemize}
\end{theorem}

\begin{proof}
(i)$\Rightarrow$(ii) 
Let $(\tilde{F}_i)_{i\in \bb{I}}$ be a net as in Lemma \ref{soat} 
and $F_i = \tilde{F}_i \circ F$, $i\in \bb{I}$. 
By Lemma  \ref{l_soat}, $(F_i)_{i\in \bb{I}}$ satisfies the conditions of (ii).

(ii)$\Rightarrow$(iii) Set $\nph_i = \cl N(F_i)$; 
by Theorem \ref{th_tra}, $\nph_i$ is a Schur $A$-multiplier, $i\in \bb{I}$. 
Since $F_i$ is finitely supported, $\nph_i$ is band finite. 
Moreover, 
$$\nph_i(s,t)(a) = \alpha_{t^{-1}} (F_i(ts^{-1})(\alpha_t(a))), \ \ \ a\in A,$$
and hence
$$\ran \alpha_s\circ \nph_i(s,e) \subseteq \ran F_i(s^{-1}) \subseteq \cl U_i, \ \ \ s\in G.$$

(iii)$\Rightarrow$(i)
Let $\Phi_i$ 
be the restriction of the map $S_{\nph_i}$ to $A\rtimes_{r,\alpha} G$, $i\in \bb{I}$.
By Proposition \ref{p_FPhi} (i) and (ii), 
$\|F - F_{\Phi_i}\|_{\rm m} \to_{i\in \bb{I}} 0$.
Since $\nph_i$ is band finite, $F_i:=F_{\Phi_i}$ is supported on a finite set, say $E_i\subseteq G$. By Proposition \ref{p_FPhi}  (v), 
$$\ran F_i(s)\subseteq \ran \nph_i(s^{-1},e) \subseteq \alpha_s(\cl U_i), \ \ \ s\in G.$$ 
Let $\cl V_i = {\rm span}\left(\cup_{s\in E_i} \alpha_s(\cl U_i)\right)$; then $\cl V_i$ is finite dimensional 
and 
$$\ran S_{F_i}\subseteq {\rm span}\{\pi(a)\lambda_s : a\in \cl V_i, s\in E_i\}.$$
Thus, $S_{F_i}$ has finite rank, $i\in \bb{I}$. 
Since the set $\ccm(A\rtimes_{r,\alpha} G)$ is closed in the completely bounded norm,
$S_F\in \frak{S}_{\rm cc}(A\rtimes_{r,\alpha} G)$.
\end{proof}

\begin{corollary}\label{c_finap}
If $A\rtimes_{r,\alpha}G$ possesses the SOAP then
every completely compact Herz-Schur $(A,G,\alpha)$-multiplier is a limit of  
$(A,G,\alpha)$-multipliers of finite rank. 
\end{corollary}

\noindent {\bf Remark.} By \cite[Theorem 5.9]{webster},  any completely compact map on a $C^*$-algebra with SOAP is a limit of finite rank maps.  Corollary \ref{c_finap} is a refinement of this statement. 
\medskip

Let $A_{\rm cb}(G)$ be the closure of 
$A(G)$ within $M^{\rm cb}A(G)$ with respect to $\|\cdot\|_{\rm m}$.  
The algebra was first introduced and studied in \cite{forrest}.  
It is a regular commutative Tauberian Banach algebra whose 
Gelfand space can be canonically identified with $G$ (see \cite[Proposition 2.2]{frs}).

It is known that $M^{\cb} A(G)$ is a dual Banach space \cite[Proposition 1.10]{ch}. 
Recall that $G$ has the {\it approximation property (AP)} \cite{haagerup-kraus}
if there exists a net $(u_i)_{i\in \bb{I}}$ of finitely supported Herz-Schur multipliers such that 
$u_i \to_{i\in \bb{I}} 1$ in the weak* topology of $M^{\rm cb}A(G)$. 
By \cite[Theorem 1.9]{haagerup-kraus} (see also \cite[Theorem 12.4.9]{brown-ozawa}, $G$ has (AP) if and only if $C_r^*(G)$ has SOAP.
We note that, by \cite[Theorem 3.6]{suzuki}, $A\rtimes_{r,\alpha} G$ has the SOAP if $A$ has the SOAP and $G$ has the AP. The SOAP for $A\rtimes_{r,\alpha} G$ was also  recently studied in \cite{mckee_turowska}.

Theorem \ref{th_ccc} has  the following immediate consequence.

\begin{corollary}\label{th_ccg}
Let $G$ be a discrete group possessing property (AP) and $u : G\to \bb{C}$ be a Herz-Schur multiplier.
The following are equivalent:
\begin{itemize}
\item[(i)] the map $S_u : C^*_r(G) \to C^*_r(G)$ is completely compact;

\item[(ii)] there exists a net $(u_i)_{i\in \bb{I}}$ of finitely supported elements of $A(G)$ such that
$\|S_u - S_{u_i}\|_{\cb}\to_{i\in \bb{I}} 0$;

\item[(iii)] there exists a net $(\nph_i)_{i\in \bb{N}}$ of band finite Schur multipliers
such that $\|S_{N(u)} - S_{\nph_i}\| \to_{i\in \bb{I}} 0$;

\item[(iv)] $u\in A_{\rm cb}(G)$.
\end{itemize}
\end{corollary}

\begin{proof}
The equivalences (i)$\Leftrightarrow$(ii)$\Leftrightarrow$(iii) follow from Theorem \ref{th_ccc}.
The equivalence (ii)$\Leftrightarrow $(iv) follows from the facts that
the finitely supported functions in $A(G)$ form a dense set in $A(G)$
and $\|S_u\|_{\rm cb} = \|u\|_{\rm m}$.
\end{proof}

\begin{remark}\label{r_MAG}
Let $G$ be a discrete group possessing (AP) and $u\in MA(G)$. 
The map $S_u$ is compact if and only if $u\in A_M(G):=\overline{A(G)}^{\|\cdot\|_{MA(G)}}$. 
\end{remark}

\begin{proof}
By \cite[Theorem 1.9]{haagerup-kraus}, there exists a net
$(u_i)_{i\in \bb{I}}$ consisting of finitely supported functions such that 
$S_{u_i}(x) \to_{i\in \bb{I}} x$, $x\in C^*_r(G)$. 
If $S_u$ is compact then $S_{u_i} S_{u}\to_{i\in \bb{I}} S_u$ in norm; thus, 
$\|u - u u_i\|_{MA(G)} \to_{i\in \bb{I}} 0$. 

Conversely, suppose that $u\in A_M(G)$. 
Since $\|\cdot\|_{MA(G)} \leq \|\cdot\|_{A(G)}$, 
there exists a net $(u_i)_{i\in \bb{I}}$ of finitely supported functions such that 
$\|u_i - u\|_{MA(G)} \to_{i\in \bb{I}} 0$. 
Thus, $\|S_{u_i} - S_u\| \to_{i\in \bb{I}} 0$; since $S_{u_i}$ has finite rank, $i\in \bb{I}$, we have that $S_u$ is 
compact. 
\end{proof}

If the group $G$ is amenable then 
$B(G) = M^{\rm cb}A(G) = MA(G)$  and the norms on these three spaces coincide, (\cite[Corollary 1.8 (ii)]{ch}). 
It was shown in \cite[Proposition 6.10]{lau} that, in this case, the map $S_u$ is compact precisely when $u\in A(G)$.
By Corollary \ref{th_ccg}, automatic complete compactness holds:

\begin{corollary}\label{disamen}
Let $u:G\to\mathbb C$ be a Herz-Schur multiplier. If $G$ is  a discrete amenable group then the following are equivalent:
\begin{itemize}
\item[(i)] the map $S_u$ is completely compact;

\item[(ii)] $u\in A(G)$.
\end{itemize}
\end{corollary}

\begin{proposition}\label{p_cnoncc} 
Let $G$ be  a discrete group containing the free group $F_\infty$ on 
infinitely many generators. If $G$ has the (AP), then there exists a multiplier $u\in MA(G)$, for which the map
$S_u$ is compact but not completely compact.
\end{proposition}

\begin{proof}
Write $H = F_\infty$, considering it as a subgroup of $G$.
We note first that
if $\varphi\in A(H)$ and $\tilde\varphi$ is the extension by zero of $\varphi$ to $G$,
then $\|\varphi\|_{A(H)}=\|\tilde\varphi\|_{A(G)}$.
In fact, 
we have $\varphi(s)=(\lambda_H(s)\xi,\eta)$ for some $\xi,\eta\in\ell^2(H)$ such that
$\|\varphi\|_{A(H)} = \|\xi\|_2\|\eta\|_2$. Considering $\ell^2(H)$ as a subspace of $\ell^2(G)$ and letting 
$\tilde\xi$ and $\tilde \eta$ be the extensions by zero to $\ell^2(G)$ of $\xi $ and $\eta$, respectively, we have that
$\tilde \varphi(s)=(\lambda_G(s)\tilde\xi,\tilde\eta)$  
and hence $\|\varphi\|_{A(G)}\leq\|\tilde\xi\|_2\|\tilde\eta\|_2=\|\varphi\|_{A(H)}$.
As the restriction map $r:A(G)\to A(H)$, $u\mapsto u|_H$, is contractive
\cite[Proposition 3.21]{eymard}, we have also
$\|\varphi\|_{A(H)}\leq\|\tilde\varphi\|_{A(G)}$, giving $\|\varphi\|_{A(H)} = \|\tilde\varphi\|_{A(G)}$.

In the proof of \cite[Theorem 2]{bozejko}, 
Bo\.{z}ejko constructed functions $\varphi_n\in A(H)$ with finite supports 
$E_n\subseteq H \subseteq G$ such that $\|\varphi_n\|_{MA(H)}=1$ but
$\|\varphi_n\|_{M^{\rm cb }A(H)}\geq C\sqrt{n}$, for some constant $C > 0$. Given $u\in A(G)$, we now have
$$\|\varphi_n u\|_{A(G)} = \|\varphi_n u|_{H}\|_{A(H)}\leq \|u|_{H}\|_{A(H)}\leq\|u\|_{A(G)};$$
for the last inequality we use the contractability of $r$.
Thus, $\|\varphi_n\|_{MA(G)} \leq 1$.
Let $\tilde{\varphi}_n$ be the extension by zero of $\varphi_n$ 
to a function on $G$, $n\in \bb{N}$, and $\psi_n = N(\tilde{\varphi}_n)$.
By \cite{bf}, $\psi_n$ is a Schur multiplier, $n\in \bb{N}$, and 
$$ \|\varphi_n\|_{M^{\rm cb}A(G)}
 = \|S_{\psi_n}\|_{\rm cb} 
 \geq\|S_{\psi_n\chi_{H\otimes H}}\|_{\rm cb}
  = \|\varphi_n\|_{M^{\rm cb}A(H)}\geq C\sqrt{n}.$$
 This shows that the norms $\|\cdot\|_{MA(G)}$ and $\|\cdot\|_{M^{\rm cb}A(G)}$ are not equivalent on $A(G)$. 
 As the completely bounded norm dominates the multiplier norm, by applying the open mapping theorem one obtains that
 $A_{\rm cb}(G) \ne  A_M(G)$. 
By Corollary  \ref{th_ccg} and Remark \ref{r_MAG}, 
there exists a compact multiplier 
which is not completely compact.
\end{proof}

We remark that, for the free group $\bb{F}_2$ on two generators, the inequality 
$A_{\rm cb}(\bb{F}_2)\ne A_M(\bb{F}_2)$ was proved in \cite[Proposition 4.2]{br_forr}.


\section{Subclasses of completely compact Herz-Schur multipliers}\label{s_ex}

In this section, we exhibit some canonical ways to construct completely compact Herz-Schur 
multipliers of dynamical systems, and describe them explicitly in an important special case. 
We assume throughout the section that $(A,G,\alpha)$ is a C*-dynamical system, where $G$ is a 
discrete group. 
A linear map $T: A\to A$ will be called \emph{$\alpha$-invariant} if 
\begin{equation}\label{eq_ia}
\alpha_t\circ T = T\circ\alpha_t, \ \ \ t\in G.
\end{equation}
Note that, if $T \in {\rm CB}(A)$ is $\alpha$-invariant then 
$$(T\otimes \id\mbox{}_{\cl B(\ell^2(G))})(\pi(a)) = \pi(T(a)), \ \ \ a\in A.$$
It follows that the map $T\otimes \id_{\cl B(\ell^2(G))}\in {\rm CB}(\cl B(\cl H))$ leaves $A\rtimes_{\alpha,r}G$ invariant, and 
hence its restriction to $A\rtimes_{\alpha,r}G$, which will be denoted by $\tilde{T}$, is well-defined.
We have that
\begin{equation}\label{eq_alpin}
\tilde{T}(\pi(a)\lambda_t) = \pi(T(a))\lambda_t, \ \ \ a\in A, t\in G.
\end{equation}
Thus, letting $F_T : G\to {\rm CB}(A)$ be the function given by 
$F_T(s) = T$, $s\in G$, we have that $\tilde{T} = S_{F_T}$, and that $F_T(s)$ is $\alpha$-invariant for every $s\in G$. 
This leads us to introducing the following classes of Herz-Schur multipliers:
$$\frak{S}^{\rm inv}(A,G,\alpha) = \left\{F\in \frak{S}(A,G,\alpha) : F(s) \mbox{ is } \alpha\mbox{-invariant for every } s\in G\right\},$$
and 
$$\frak{S}^{\rm inv}_{\rm cc}(A,G,\alpha) = \frak{S}^{\rm inv}(A,G,\alpha)\cap \frak{S}_{\rm cc}(A,G,\alpha).$$

We remark that $\frak{S}^{\rm inv}(A,G,\alpha)$ should not be confused with $\frak{S}_{\rm inv}(G,G,A)$ which appears in \cite{mtt} and denotes the so-called invariant Schur-$A$-multipliers (see the remark after Theorem \ref{th_tra}); while invariance is a property of the map 
$\nph : G\times G\to {\rm CB}(A)$, $\alpha$-invariance is a property of the maps $F(s)$, $s\in G$.

We will denote by $Z(A)$ the centre of $A$, and by $Z(A)^+$ the cone of its positive elements. 
Recall that the action of a discrete group $G$ on $A$ is 
\emph{amenable} \cite[Definition 4.3.1]{brown-ozawa}
if there exists a net $(\xi_i)_{i\in \bb{I}}$ of finitely supported functions 
$\xi_i : G\to Z(A)^+$, such that $\sum_{i\in \bb{I}}\xi_i(t)^2 =1_A$ and
$$\left\|\sum_{s\in G} \left(\xi_i(s) - \alpha_t(\xi_i(t^{-1}s))\right)^2\right\| \to_{i\in \bb{I}} 0$$
for every $t\in G$.

\begin{theorem}\label{ccmult2}
Let $\alpha$ be an amenable action of a discrete group $G$ on a unital C*-algebra $A$.
Let $u\in M^{\rm cb}A(G)$, $T\in {\rm CB}(A)$ be $\alpha$-invariant and
$F_{T,u} : G\to {\rm CB}(A)$ be given by 
$F_{T,u}(t)(a) = u(t)T(a)$, $a\in A$, $t\in G$. The following hold:
\begin{itemize}
\item[(i)] $F_{T,u}\in \frak{S}^{\rm inv}(A,G,\alpha)$;

\item[(ii)] $F_{T,u}\in \frak{S}_{\rm cc}^{\rm inv}(A,G,\alpha)$ if $u\in A_{\rm cb}(G)$ and $T\in {\rm CC}(A)$;

\item[(iii)] $F_{T,u}\in \overline{{\rm F}(A\rtimes_{\alpha,r}G)\cap \frak{S}(A,G,\alpha)}^{\rm cb}$
if $u\in A_{\rm cb}(G)$ and $T\in \overline{{\rm F}(A)}^{\rm cb}$.
\end{itemize}
\end{theorem}

\begin{proof}
(i) 
Let $\tilde{u} : G\to {\rm CB}(A)$ be given by $\tilde{u}(t) = u(t)\id_A$. 
The map $S_{\tilde{u}}$ is the restriction to $A\rtimes_{\alpha,r}G$ of the 
map $\id_{\cl B(H)}\otimes S_u\in {\rm CB}(\cl B(H)\otimes_{\min} C_r^*(G))$ and hence
$\tilde{u}\in \frak{S}(A,G,\alpha)$.
Since $\tilde{T} = S_{F_T}$ and $F_T\in \frak{S}^{\rm inv}(A,G,\alpha)$, we conclude that 
$$F_{T,u} = \tilde{T} \circ S_{\tilde{u}} \in \frak{S}^{\rm inv}(A,G,\alpha).$$ 

(ii) 
Since the space of finitely supported functions in $A(G)$ is dense in $A(G)$ \cite{eymard}
and $\|\cdot\|_{\rm m}\leq \|\cdot\|_{A}$, given $\varepsilon > 0$, there exists $v\in A(G)$ supported in a finite set 
$E\subseteq G$ such that $\|u-v\|_{\rm m} < \varepsilon/\|T\|_{\rm cb}$.
Let $\cl F\subseteq A$ be a finite dimensional space such that
$\text{dist}(T^{(n)}(a), M_n(\cl F)) < \varepsilon/|E|$ for every $a\in M_n(A)$, $\|a\|\leq 1$, and every $n\in\mathbb N$.

Set 
$$\cl Y = \left\{\sum_{s\in E}\pi(a_s)\lambda_s: a_s\in \cl F\right\} \mbox{ and }
\cl Z = \left\{\sum_{s\in E}\pi(a_s)\lambda_s: a_s\in A\right\};$$ 
we have that $\cl Y$ is a finite-dimensional subspace of $\cl Z$. 
Let $x = [x_{k,l}]\in M_n(\cl Z)$, $\|x\|\leq 1$,  and write $L_s = \text{diag}(\lambda_s,\ldots,\lambda_s)\in M_n(\cl Z)$.  
Then 
$$[x_{k,l}] = \sum_{s\in E}  (\pi\circ \cl E)^{(n)}\left([x_{k,l}]L_{s^{-1}}\right)L_s;$$ 
as $\cl E$ is completely contractive, 
$\left\|\cl E^{(n)}\left([x_{k,l}]L_{s^{-1}}\right)\right\|\leq 1$.
For each $s\in E$, choose $y_s = [y_{k,l}^s]\in M_n(\cl F)$ such that 
\begin{equation}\label{eq_TL}
\left\|T^{(n)}\left(\cl E^{(n)}\left([x_{k,l}]L_{s^{-1}}\right)\right) - y_s\right\| < \varepsilon/|E|
\end{equation}
and set $y = \sum_{s\in E} \pi^{(n)}(y_s) L_s\in M_n(\cl Y)$.  
By (\ref{eq_TL}),
$$\left\|\tilde{T}^{(n)}\left([x_{k,l}]\right)-y\right\|
= \left\|\sum_{s\in E}\left[T\left(\cl E\left(x_{k,l}\lambda_{s^{-1}}\right)\right)\lambda_s-y_{k,l}^s\lambda_s\right]\right\|
\leq \varepsilon.$$
Thus, $\tilde{T}|_{\cl Z}$ is completely compact. 
Since the image of $S_v$ is in ${\cl Z}$ we obtain that $\tilde{T}\circ S_v$ is completely compact.
The statement now follows from the inequalities
$$\|\tilde{T}\circ S_v - \tilde{T}\circ S_u\|_{\rm cb}
\leq 
\|\tilde{T}\|_{\cb} \|S_v - S_u\|_{\rm cb} = \|\tilde{T}\|_{\rm cb} \|u - v\|_{\rm m}
\leq 
\varepsilon $$
and the fact that the space of completely compact maps is closed with respect to the completely bounded norm.

(iii) Let
$(T_k)_{k\in \bb{N}}\subseteq {\rm F}(A)$ be a sequence such that $\|T_k - T\|_{\cb}\to_{k\to\infty} 0$. 
We follow the idea in the proof of \cite[Corollary 4.6]{mstt}.
Let $(\xi_i)_{i\in I}$ be a net as in definition of amenable action and
the maps $F_{i,k}(s) : A\to A$, $s\in G$, $i\in \bb{I}$, $k\in \bb{N}$, 
be given by
\begin{equation}\label{eq_Fij}
F_{i,k}(s)(a) = 
\sum_{q\in G} \xi_i(q)\alpha_q(T_k(\alpha_q^{-1}(a))\alpha_s(\xi_i(s^{-1}q)), \ \ \ a\in A.
\end{equation}
Note that if $E_i\subseteq G$ is a finite set with 
$\supp\xi_i\subseteq E_i$ then the sum on the right hand side of (\ref{eq_Fij}) 
is over the (finite) set $E_i\cap sE_i$, for every $k\in \bb{N}$.
We have
\begin{eqnarray*}
\cl N(F_{i,k})(s,t)(a) 
& = & 
\alpha_{t^{-1}}\left(\sum_{q\in G} \xi_i(q)\alpha_q(T_k(\alpha_q^{-1}(\alpha_t(a)))\alpha_{ts^{-1}}(\xi_i(st^{-1}q))\right)\\
& = & 
\sum_{q\in G} \alpha_{t^{-1}}(\xi_i(q))\alpha_{t^{-1}q} (T_k(\alpha_{q^{-1}t}(a)))\alpha_{s^{-1}}(\xi_i(st^{-1}q))\\
& = & 
\sum_{p\in G} \alpha_{t^{-1}}(\xi_i(tp))\alpha_{p} (T_k(\alpha_{p^{-1}}(a)))\alpha_{s^{-1}}(\xi_i(sp)).
\end{eqnarray*}
Since $T_k : A\to A$ is a completely bounded map, by Haagerup-Paulsen-Wittstock Theorem,
there exist a Hilbert space $H_{p,k}$, a *-representation $\pi_{p,k} : A\to\cl B(H_{p,k})$ 
and bounded operators $V_{p,k}$, $W_{p,k} : H \to H_{p,k}$, such that
$$\alpha_p(T_k(\alpha_p^{-1}(a))) = W_{p,k}^*\pi_{p,k}(a)V_{p,k}, \ \ \ a\in A,$$
and $\|T_k\|_{\rm cb} = \|\alpha_p\circ T_k\circ\alpha_p^{-1}\|_{\cb} =  \|V_{p,k}\|\|W_{p,k}\|$. Renormalising we may assume that $ \|V_{p,k}\|=\|W_{p,k}\|=\|T_k\|_{\rm cb}^{1/2}$ for all $p$. 

Set $\rho_k := \oplus_{p\in G}\pi_{p,k}$ and let ${\bf V}_{i,k}(s) : H\to\oplus_{p\in G} H_{p,k}$ and
${\bf W}_{i,k}(t) : H\to\oplus_{p\in G} H_{p,k}$ be the column operators given by
$${\bf V}_{i,k}(s) = (V_{p,k}\alpha_s^{-1}(\xi_i(sp)))_{p\in G} \ \text{ and } \ 
{\bf W}_{i,k}(t) = (W_{p,k}\alpha_t^{-1}(\xi_i(tp)))_{p\in G}.$$
Then
$$\cl N(F_{i,k})(s,t)(a) = {\bf W}_{i,k}^*(t)\rho_k(a){\bf V}_{i,k}(s)$$
and
\begin{eqnarray*}
\|{\bf W}_{i,k}(t)\|\|{\bf V}_{i,k}(s)\|
& = &
\left\|\sum_{p\in G}\alpha_{t^{-1}}(\xi_i(tp))W_{p,k}^*W_{p,k}\alpha_{t^{-1}}(\xi_i(tp))\right\|^{1/2}\\
&&
\times \left\|\sum_{p\in G}\alpha_{s^{-1}}(\xi_i(sp))V_{p,k}^*V_{p,k}\alpha_{s^{-1}}(\xi_i(sp))\right\|^{1/2}\\
\\&\leq &
\|T_k\|_{\rm cb} \left\|\sum_{p\in G}\alpha_{s^{-1}}(\xi_i(sp))^2\right\|\\
&= &
\|T_k\|_{\rm cb} \left\|\sum_{p\in G}\xi_i(p)^2\right\| = \|T_k\|_{\cb}.
\end{eqnarray*}
Hence, by \cite[Theorem 2.6]{mtt},
$\cl N(F_{i,k})$  is a Schur $A$-multiplier, for all $i\in \bb{I}$ and all $k\in \bb{N}$. 
By Theorem \ref{th_tra}, $F_{i,k}$ is a Herz-Schur $(A,G,\alpha)$-multiplier; 
moreover, as the sequence  $(\|T_k\|_{\cb})_{k\in \bb{N}}$ is bounded, there exists $C > 0$ such that 
$\|S_{F_{i,k}}\|_{\cb} \leq C$, $i\in \bb{I}$, $k\in \bb{N}$.

Next we prove that $\|F_{i,k}(s) - T\|_{\cb}\to 0$ for each $s\in G$.
As $T$ is $\alpha$-invariant, if $a = [a_{r,l}]\in M_n(A)$ then
\begin{eqnarray*}
&&\|F_{i,k}(s)^{(n)}(a)-T^{(n)}(a)\|\\
& = &
\left\|\sum_{p\in G}[\xi_i(p)\alpha_s(\xi_i(s^{-1}p))\alpha_p(T_k(\alpha_p^{-1}(a_{r,l})))-
\alpha_p(T(\alpha_p^{-1}(a_{r,l})))]\right\|\\
& \leq &
\sum_{p\in G} \left\|\xi_i(p)\alpha_s(\xi_i(s^{-1}p))\right\| \left\|[\alpha_p(T_k(\alpha_p^{-1}(a_{r,l})) - 
\alpha_p(T(\alpha_p^{-1}(a_{r,l})))]\right\|\\
& + &
\left\|\sum_{p\in G} \xi_i(p)\alpha_s(\xi_i(s^{-1}p))-1\right\| \left\|T^{(n)}(a)\right\|\\
& \leq &
\left\|\sum_{p\in G}\xi_i(p)^2\right\| \left\|T_k - T\right\|_{\cb}
\|a\| + \left\|\sum_{p\in G} \xi_i(p)\alpha_s(\xi_i(s^{-1}p))-1\right\| \|T\|_{\rm cb}\|a\|,
\end{eqnarray*}
where in the last line we used the Cauchy-Schwarz inequality and the fact $\xi_i(p) \in Z(A)^+$, $p\in G$.
By \cite[Lemma 4.3.2]{brown-ozawa}, $\|\sum_p\xi_i(p)\alpha_s(\xi_i(s^{-1}p)) - 1\|\to_{i\in \bb{I}} 0$; 
the facts that $\|T_k - T\|_{\cb}\to_{k\to\infty} 0$ and $\sum_{p\in G}\xi_i(p)^2 = 1$
now imply that 
\begin{equation}\label{eq_Fik0}
\|F_{i,k}(s)-T\|_{\cb}\to_{(i,k)\in \bb{I}\times\bb{N}} 0, \ \ \ s\in G.
\end{equation}
Since $T_k\in {\rm F}(A)$ and the maps $F_{i,k}$ are finitely supported, the map
$S_{F_{i,k}}$ on $A\rtimes_{\alpha,G} G$ has finite rank.  
In order to prove the statement, it hence suffices to show that
\begin{equation}\label{cbconv}
\| S_{F_{i,k}}\circ S_u - S_{F_{T,u}}\|_{\cb}\to_{(i,k)\in \bb{I}\times\bb{N}} 0.
\end{equation}

 Let $\varepsilon > 0$. Since $u\in A_{\rm cb}(G)$, there exists 
 $v\in A(G)$ with finite support $E$ such that 
 $$\|S_u - S_v\|_{\cb} \leq \|u - v\|_{A} \leq \varepsilon.$$
Set $\cl Z = \left\{\sum_{s\in E}\pi(a_s)\lambda_s: a_s\in A\right\}$ and let $x = [x_{r,l}]\in M_n(\cl Z)$. 
As in the proof of (ii), write $L_s$ for $\text{diag}(\lambda_s,\ldots,\lambda_s)\in M_n(\cl Z)$. 
Then $x = \sum_{s\in E}(\pi\circ\cl E)^{(n)}(x L_{s^{-1}})L_s$ and
\begin{eqnarray*}
\|S_{F_{i,k}}^{(n)}(x) - S_T^{(n)}(x)\| = \|\sum_{s\in E}[\pi((F_{i,k}(s) - T)(\cl E(x_{r,l}\lambda_{s^{-1}})))\lambda_s]\|\\
\leq \sum_{s\in E}\|[(F_{i,k}(s) - T)(\cl E(x_{r,l}\lambda_{s^{-1}})]\|\leq \sum_{s\in E}\|(F_{i,k}(s) - T)^{(n)}\|
\|[\cl E(x_{r,l}\lambda_{s^{-1}})]\|\\
\leq \sum_{s\in E}\|(F_{i,k}(s) - T)^{(n)}\| \|x\|\leq |E| \|x\| (\text{max}_{s\in E} \|F_{i,k}(s)^{(n)} - T^{(n)}\|).
\end{eqnarray*}
By (\ref{eq_Fik0}), 
$$\|(S_{F_{i,k}} - S_T)|_{\cl Z}\|_{\cb}\to_{(i,k) \in \bb{I}\times\bb{N}} 0.$$

We now have
\begin{eqnarray*}
& & \|S_{F_{i,k}}\circ S_u - S_{F_{T,u}}\|_{\cb}\\
& \leq & 
\|S_{F_{i,k}}\circ(S_u - S_v)\|_{\cb} + \|S_{F_{i,k}}\circ S_v - S_T \circ S_v\|_{\cb} + \|S_T \circ (S_v - S_u)\|_{\cb}\\ 
& \leq & 
\|(S_{F_{i,k}} - S_T)|_{\cl Z}\|_{\cb}\|v\| + \varepsilon(\|S_{F_{i,k}}\|_{\cb} + \|S_T\|_{\cb}),
\end{eqnarray*}
which implies (\ref{cbconv}).
\end{proof}

Theorem \ref{ccmult2} exhibits a large class of elements of $\frak{S}^{\rm inv}_{\rm cc}(A,G,\alpha)$. 
In the next theorem, we provide a precise description of the latter class of Herz-Schur multipliers in the 
case of the irrational rotation algebra. 
Let $\theta\in\mathbb R$ be irratrional and 
$\alpha : \bb{Z}\to {\rm Aut}(C(\mathbb T))$ be given by 
$$\alpha_n(f)(z)=f(e^{2\pi in\theta}z), \ \ \ f\in C(\mathbb T), z\in \mathbb T.$$
Let $M(\mathbb T)$ denote the Banach algebra of all complex Borel measures on the unit circle $\bb{T}$, 
and note that $L^1(\bb{T})$ is a (closed) ideal of $M(\mathbb T)$. 
For $\mu\in M(\mathbb T)$, let 
$T_{\mu} : C(\mathbb T)\to C(\mathbb T)$ be the completely bounded map given by $T_\mu(f) = \mu\ast f$.
Note that $T_{\mu}$ is $\alpha$-invariant; indeed,
\begin{eqnarray*}
\alpha_n(\mu\ast f)(z)&=&(\mu\ast f)(e^{2\pi in\theta}z)=\int f(e^{2\pi in\theta}w^{-1}z)d\mu(w)\\
&=&\int \alpha_n(f)(w^{-1}z)d\mu(w)=\mu\ast\alpha_n(f)(z).
\end{eqnarray*}

\begin{theorem}\label{th_eqirrr}
The functions $F_{T_f,u}$, where $f\in L^1(\bb{T})$ and $u\in A(\bb{Z})$, have a dense linear span in 
$\frak{S}^{\rm inv}_{\rm cc}(A,G,\alpha)$.
\end{theorem}

\begin{proof}
Let $F$ be a 
completely compact Herz-Schur $(C(\mathbb T),\mathbb Z,\alpha)$-multiplier 
such that $F(n)\circ \alpha_m = \alpha_m\circ F(n)$ for all $m,n\in\mathbb Z$. 
We show that $F(n) = T_{f_n}$ for some $f_n\in L^1(\mathbb T)$. 
In fact, let $\Phi_n : M(\mathbb T)\to M(\mathbb T)$ be the dual map of $F(n)$. 
As $\alpha_n(f)(z) = f(e^{2\pi in\theta}z) = (\delta_{e^{-2\pi in\theta}}\ast f)(z)$, where $\delta_s$ 
is the point mass measure at $s\in\mathbb T$, 
we obtain
\begin{eqnarray*}
\left\langle \Phi_n(\delta_{e^{2\pi im\theta}}\ast\mu), f\right\rangle 
& = &
\left\langle \delta_{e^{2\pi im\theta}}\ast\mu, F(n)(f)\right\rangle 
= 
\left\langle \mu,\alpha_{-m}(F(n)(f))\right\rangle\\
& = & 
\left\langle \mu,F(n)(\alpha_{-m}(f))\right\rangle
=
\left\langle \delta_{e^{2\pi im\theta}}\ast\Phi_n(\mu),f\right\rangle,
\end{eqnarray*}
giving 
$$\Phi_n(\delta_{e^{2\pi im\theta}}\ast\mu)=\delta_{e^{2\pi im\theta}}\ast\Phi_n(\mu), 
\ \ m,n\in \bb{Z}, \mu\in M(\mathbb T).$$ 
In particular, $\Phi_n(\delta_{e^{2\pi im\theta}})=\delta_{e^{2\pi im\theta}}\ast\Phi_n(\delta_1)$. 
Using the weak* continuity of $\Phi_n$, 
the density of $\{e^{2\pi im\theta}: m\in\mathbb Z\}$ in $\mathbb T$ and the fact that 
every measure $\mu$ is a weak* limit of linear combinations 
of point mass measures,
we obtain
$$\Phi_n(\mu) = \mu\ast\Phi_n(\delta_1), \ \ \ \mu\in M(\mathbb T).$$ 
It now follows that, if $\mu_n = \Phi_n(\delta_1)$ then $F(n)(f) = \mu_n\ast f$, $f\in C(\mathbb T)\simeq C_r^*(\mathbb Z)$.
By Corollary \ref{c_Fscc}, $F(n)$ is completely compact for every $n\in \bb{Z}$; by 
Corollary \ref{disamen}, $\mu_n\in L^1(\mathbb T)$.

Fix $F\in \frak{S}^{\rm inv}_{\rm cc}(A,G,\alpha)$; then $F(n)$ is $\alpha$-invariant for every $n\in \bb{Z}$. 
We show that 
$S_F\in\overline{[S_u\circ S_{F(n)}: u\in A(\mathbb Z), n\in\mathbb Z]}^{\|\cdot\|_{\cb}}$. 
By the amenability of $\mathbb Z$, there exists a bounded sequence 
$(u_i)_{i\in \bb{N}}\subseteq A(\mathbb Z)$
of finitely supported functions such that  $S_{u_i}\to \id$ in the strong point norm topology 
(see \cite[Theorem 1.12, Theorem 1.9]{haagerup-kraus}).
Since $u_i$ is finitely supported, 
$$S_{u_i}\circ S_F\in {\rm span}\{S_u\circ S_{F(n)}: u\in A(\mathbb Z), n\in\mathbb Z\}.$$ 
Since $S_F$ is completely compact, 
Lemma \ref{l_soat} implies $\|S_{u_i}\circ S_F - S_F\|_{\cb}\to_{i\to\infty} 0$.
\end{proof}


\section{Herz-Schur multipliers of $\cl K$}\label{s_K}

Let $G$ be a discrete group 
and $\alpha : G\to {\rm Aut}(c_0(G))$ be the homomorphism given by 
$\alpha_t(f)(s) = f(t^{-1}s)$, $s\in G$, $f\in c_0(G)$. 
For $a\in c_0(G)$, we write $M_a$ for the operator on $\ell^2(G)$ given by $(M_a\xi)(s) = a(s) \xi(s)$,
$\xi\in \ell^2(G)$, $s\in G$, and 
let
$$\cl C = \{M_a : a\in c_0(G)\};$$
the map $\iota : c_0(G)\to \cl C$ given by $\iota(a) = M_a$ is a *-isomorphism.
By abuse of notation, we write $\alpha_t$ for the corresponding automorphism of $\cl C$.
Recall that, for $t\in G$, we denote by $\lambda_t^0$ 
the left regular unitary acting on $\ell^2(G)$.
The pair $(\iota,\lambda^0)$ of representations is covariant, 
and hence gives rise to a
(faithful) representation $\iota\rtimes\lambda^0  : \cl C\rtimes_{\alpha,r} G \to \cl B(\ell^2(G))$.
According to the Stone-von Neumann Theorem
\cite[Theorem 4.24]{w}, the image of $\iota\rtimes \lambda^0$ coincides with the C*-algebra $\cl K$ 
of all compact operators on $\ell^2(G)$.
Thus, the Herz-Schur $(\cl C,G,\alpha)$-multipliers give rise, in a canonical way, to 
a certain class of completely bounded maps on $\cl K$. The aim of this section is to formalise this 
correspondence and examine the complete compactness of the resulting maps on $\cl K$. 

Set $\cl B = (\pi\rtimes\lambda)(\cl C\rtimes_{\alpha,r} G)$; 
thus $\cl B$ is a C*-subalgebra of $\cl B(\ell^2(G\times G))$.
By \cite[Theorem 4.24]{w}, 
the map
$$\theta : \pi(a)\lambda_t \longrightarrow M_a \lambda_t^0, \ \ \ a\in \cl C, t\in G,$$
is a *-isomorphism from $\cl B$ onto $\cl K$.
For an element $\Phi \in \cbm(\cl K)$, let 
$\widetilde{\Phi} = \theta^{-1}\circ \Phi \circ \theta$; thus,
$\widetilde{\Phi}\in \cbm(\cl B)$, and the
correspondence $\Phi \longrightarrow \widetilde{\Phi}$ between $\cbm(\cl K)$ and $\cbm(\cl B)$ is bijective.
For a Herz-Schur $(\cl C,G,\alpha)$-multiplier $F$,
let $\Phi_F = \theta\circ S_F \circ \theta^{-1}$
note that $\Phi_F\in {\rm CB}(\cl K)$ and
$$\Phi_F(M_a \lambda_t^0) = M_{F(t)(a)} \lambda_t^0, \ \ \ a\in \cl C, t\in G.$$
We call the maps of the form $\Phi_F$ the \emph{Herz-Schur $\cl K$-multipliers}.
We let
$$\frak{S}(\cl K) = \{\Phi_F : F \mbox{ is a Herz-Schur } (\cl C,G,\alpha)\mbox{-multiplier}\},$$
and
$\frak{S}_{\rm cc}(\cl K) = \frak{S}(\cl K)\cap {\rm CC}(\cl K)$.
We note that the elements of $\frak{S}(\cl K)$ are precisely those completely bounded maps on $\cl K$ 
which leave its diagonals globally invariant, when elements of $\cl K$ are viewed as $G\times G$-matrices.

We recall that, if $K$ is a Hilbert space and $\cl A\subseteq \cl B(K)$ 
is a C*-algebra, the Haagerup tensor product $\cl A\otimes_{\rm h}\cl A$ 
consists of (convergent) series $u = \sum_{i=1}^{\infty} a_i\otimes b_i$, where $(a_i)_{i\in \bb{N}}\subseteq \cl A$
and $(b_i)_{i\in \bb{N}}\subseteq \cl A$ are sequences for which the series 
$\sum_{i=1}^{\infty} a_i a_i^*$ and $\sum_{i=1}^{\infty} b_i^* b_i$ converge in norm. 
Each such $u$ gives rise to a completely bounded map $\Gamma_u : \cl B(K)\to \cl B(K)$ given by 
$$\Gamma_u(x) = \sum_{i=1}^\infty a_ixb_i, \ \ \ x\in \cl B(K).$$
The following fact was noted in  \cite[Corollary 3.6]{jltt}.

\begin{theorem}\label{complcomp}
Any completely compact map on $\cl K$ has the form $\Gamma_u$ for some element
$u \in \cl K\otimes_{\rm h}\cl K$.
\end{theorem}

\begin{theorem}\label{th_corr}
The mapping
\begin{equation}\label{eq_PhiFPhi}
\Phi \longrightarrow \Phi_{F_{\widetilde{\Phi}}}
\end{equation}
is a linear contractive surjection
\begin{itemize}
\item[(i)] from $\cbm(\cl K)$ onto $\frak{S}(\cl K)$;

\item[(ii)] from ${\rm CC}(\cl K)$ onto $\frak{S}_{\rm cc}(\cl K)$.
\end{itemize}
\end{theorem}

\begin{proof}
(i) By Proposition \ref{p_FPhi} (i), if $\Phi\in \cbm(\cl K)$ then $F_{\widetilde{\Phi}}$
is a Herz-Schur $(\cl C,G,\alpha)$-multiplier and $\|F_{\widetilde{\Phi}}\|_{\rm m}\leq \|\Phi\|_{\rm cb}$. 
Thus,
$\Phi_{F_{\widetilde{\Phi}}}\in \frak{S}(\cl K)$
and $\|\Phi_{F_{\widetilde{\Phi}}}\|_{\rm cb}\leq \|\Phi\|_{\rm cb}$.
On the other hand, if $F$ is a Herz-Schur $(\cl C,G,\alpha)$-multiplier then, by Proposition \ref{p_FPhi} (ii),
$F_{S_F} = F$
and hence $\Phi_F$ is the image of itself under the map (\ref{eq_PhiFPhi}).
The linearity of the map (\ref{eq_PhiFPhi}) is straighforward.

(ii) 
Let $\Phi\in {\rm CC}(\cl K)$. Since $\cl K$ has SOAP, 
Lemma \ref{l_soat} shows that
$\Phi$ can be approximated in the completely bounded norm by finite rank maps 
$\Phi_n : \cl K\to \cl K$, $n\in \bb{N}$.  
Clearly, $\tilde \Phi_n$ has finite rank, $n\in \bb{N}$. 
As in the proof of Lemma \ref{soat}, we can assume that 
the maps $\tilde\Phi_{n}$ have 
range in $\text{span}\{\pi(a)\lambda_s: s\in G, a\in \cl C\}$. By Proposition  \ref{p_FPhi} (iv),
 $S_{F_{\tilde\Phi_n}}$ 
 has finite rank and hence so has 
 $\Phi_{F_{\tilde\Phi_n}}$. Moreover, by Proposition \ref{p_FPhi} (i), $\|F_{\tilde\Phi}-F_{\tilde\Phi_n}\|_{\rm m}\to 0$, 
 giving $\|\Phi_{F_{\tilde\Phi}} - \Phi_{F_{\tilde\Phi_n}}\|_{\rm cb} \to_{n\to \infty} 0$.
Hence $\Phi_{F_{\tilde\Phi}}$ is completely compact. The rest is similar to (i).
\end{proof}

For $b\in\cl K$ and $s\in G$, let $b_s = \lambda_s^0b\lambda_{s^{-1}}^0$.
For $u = \sum_{i=1}^\infty a_i\otimes b_i\in\cl K\otimes_{\rm h}\cl K$, write
$u_s = \sum_{i=1}^\infty a_i\otimes (b_i)_s$; 
it is clear that $u_s$ is a well-defined element of $\cl K\otimes_{\rm h}\cl K$.

\begin{theorem}\label{th_KK}
Let $F : G\to {\rm CB}(\cl C)$. The following are equivalent:
\begin{itemize}
\item[(i)] $F$ is a completely compact Herz-Schur $(\cl C, G,\alpha)$-multiplier;
\item[(ii)] there exists $u \in \cl K\otimes_{\rm h}\cl K$ such that
$\Gamma_{u_s}(a) = F(s)(a)$, for all $a \in \cl C$ and $s\in G$.
\end{itemize}
\end{theorem}

\begin{proof} 
(i)$\Rightarrow$(ii) 
The map 
$M_a\lambda_s^0\mapsto M_{F(s)(a))}\lambda_s^0$ extends to a completely compact map on $\cl K$. 
By Theorem \ref{complcomp}, there exists
$u = \sum_{i=1}^\infty a_i\otimes b_i\in\cl K\otimes_{\rm h}\cl K$ such that
$$M_{F(s)(a)}\lambda_s^0 
= \sum_{i=1}^\infty a_iM_a \lambda_s^0 b_i 
= \sum_{i=1}^\infty a_iM_a\lambda_s^0 b_i (\lambda_s^0)^*\lambda_s^0, \ \ \ a\in \cl C.$$ 
Thus, $\cl C$ is an invariant subspace for $\Gamma_{u_s}$ and $\Gamma_{u_s}(a) = F(s)(a)$, $a\in \cl C$.

(ii)$\Rightarrow$(i) 
By the previous paragraph, $\Gamma_u(M_a\lambda_s^0) = \Gamma_{u_s}(M_a)\lambda_s^0$. 
Thus, $\theta^{-1}\circ \Gamma_u \circ \theta$ coincides with $S_F$ and so 
$F$ is a completely compact Herz-Schur multiplier.  
\end{proof}

Suppose that $F$ is a Herz Schur $(\cl C,G,\alpha)$-multiplier of the form 
$F(s)(M_a) = M_{h_s a}$, where $h_s : G\to \bb{C}$ is a function, $s\in G$. 
We can associate with it the function $\psi : G\times G\to \bb{C}$, given by 
$\psi(s,t) = h_s(t)$. 
This is the identification made the next statement.

\begin{corollary} 
Let $\nph : G\times G \to \bb{C}$ be a compact Schur multiplier.
Then the function $\psi : G\times G \to \bb{C}$, given by 
$\psi(s,t) = \nph(t,s^{-1}t)$, is a completely compact Herz Schur 
$(\cl C,G,\alpha)$-multiplier.
\end{corollary}

\begin{proof}
By \cite{hladnik}, there exists an element $u = \sum_{i=1}^\infty a_i\otimes b_i \in c_0(G)\otimes_{\rm h} c_0(G)$ such that 
$\nph(p,q) = \sum_{i=1}^\infty a_i(p) b_i(q)$, $p,q\in G$.
By the injectivity of the Haagerup tensor product \cite[Proposition 9.2.5]{er}, $u\in \cl K\otimes_{\rm h}\cl K$.
Note that 
$$\Gamma_{u_s}(M_a)(t) = \sum_{i=1}^\infty a_i(t) b_i(s^{-1}t)a(t), \ \ \ a\in c_0(G), t\in G;$$
in other words, 
$\Gamma_{u_s}(M_a) = M_{h_s a}$, where $h_s(t) = \nph(t,s^{-1}t)$, $t\in G$.
The conclusion follows from Theorem \ref{th_KK}. 
\end{proof}

\section{Some remarks and open questions}

In Corollary \ref{disamen}, we showed that the amenability 
of a discrete group $G$ is a sufficient condition 
for automatic complete compactness: every compact multiplier is in this case completely compact. 
For such automatic complete compactness it suffices, instead of amenability, 
to assume that the completely bounded multiplier norm is equivalent to the multiplier norm.  
By a result of V. Losert \cite{losert}, there exist non-amenable groups such that $M^{\rm cb}A(G) = MA(G)$, 
for instance $SL(2,\mathbb R)$. However we do not know whether there exists a {\it discrete} group with this property.  

In Proposition \ref{p_cnoncc}, we exhibited an example
of a multiplier $u\in MA(G)$, for which the map
$S_u$ is compact but not completely compact. 
This multiplier however
may not be completely bounded, as we only guarantee the {\it boundedness} of $S_u$ if $u\in MA(G)$.
We do not know if 
there exists a completely bounded compact multiplier which is not completely compact.

Finally, in Corollary \ref{th_ccg}, 
we showed that if 
$G$ is a discrete group possessing property (AP) then 
every completely compact multiplier on $G$ is 
the limit, in the completely bounded norm, of 
{\it finitely supported} multipliers. 
It would be interesting to know if (AP) is in fact 
equivalent to the latter approximation property; 
we do not know if this holds true. 

\bigskip


\noindent {\bf Acknowledgements. }
I.T. was partially supported by Simons Foundation grant 708084.
L.T.\ acknowledges the support and hospitality  at  Queen's University Belfast  during several visits in 2018 and 2019. The authors would like to thank Jason Crann 
for stimulating discussions. 
They are grateful to the anonymous referees for a number of comments and suggestions which helped improving the exposition.

\end{document}